\documentclass[11pt]{article}

\usepackage{amsmath, amssymb, amsthm, verbatim,enumerate,bbm,color} 
\numberwithin{equation}{section}

\usepackage{mathrsfs}
\usepackage{enumitem}

\usepackage[backref,colorlinks,citecolor=blue,bookmarks=true]{hyperref}

\usepackage{tikz}

\title{A new approach for the Brown-Erd\H{o}s-S\'os problem}

\author{Asaf Shapira \thanks{School of Mathematics, Tel Aviv University, Tel Aviv 69978, Israel. Email: asafico$@$tau.ac.il. Supported in part by ERC Consolidator Grant 863438 and NSF-BSF Grant 20196.} \and
Mykhaylo Tyomkyn
	\thanks{Department of Applied Mathematics, Charles University. Email: tyomkyn$@$kam.mff.cuni.cz.
Supported in part by ERC Synergy Grant DYNASNET 810115 and GA\v{C}R Grant 22-19073S.}}

\date{\today}
\allowdisplaybreaks

\theoremstyle{plain}
\newtheorem{theorem}{Theorem}[section]
\newtheorem{lemma}[theorem]{Lemma}
\newtheorem{claim}[theorem]{Claim}

\newtheorem{observation}[theorem]{Observation}

\newtheorem{conjecture}[theorem]{Conjecture}

\makeatletter
\def\moverlay{\mathpalette\mov@rlay}
\def\mov@rlay#1#2{\leavevmode\vtop{%
   \baselineskip\z@skip \lineskiplimit-\maxdimen
   \ialign{\hfil$\m@th#1##$\hfil\cr#2\crcr}}}
\newcommand{\charfusion}[3][\mathord]{
    #1{\ifx#1\mathop\vphantom{#2}\fi
        \mathpalette\mov@rlay{#2\cr#3}
      }
    \ifx#1\mathop\expandafter\displaylimits\fi}
\makeatother

\makeatletter
\renewenvironment{proof}[1][\proofname]
{\par\pushQED{\qed}
	\normalfont\topsep6\p@\@plus6\p@\relax\trivlist
	\item[\hskip\labelsep\bfseries#1\@addpunct{.}]
	\ignorespaces}
{\popQED\endtrivlist\@endpefalse}
\makeatother

\newcommand{\A}{\mathcal A}
\newcommand{\B}{\mathcal B}

\newcommand{\C}{\mathcal C}

\newcommand{\F}{\mathcal F}

\newcommand{\G}{\mathcal G}

\newcommand{\Z}{\mathcal Z}

\RequirePackage{color}\definecolor{RED}{rgb}{1,0,0}\definecolor{BLUE}{rgb}{0,0,1} 

\begin{document}
\date{}
\maketitle

\begin{abstract}

The celebrated Brown-Erd\H{o}s-S\'os conjecture
states that for every fixed $e$, every $3$-uniform hypergraph with $\Omega(n^2)$ edges contains
$e$ edges spanned by $e+3$ vertices. Up to this date all the approaches towards resolving this problem relied
on highly involved applications of the hypergraph regularity method, and yet they supplied only approximate
versions of the conjecture, producing $e$ edges spanned by $e+O(\log e/\log \log e)$ vertices.

In this short paper we describe a completely different approach, which reduces the problem to a variant of another well-known conjecture in extremal graph theory. A resolution of the latter would resolve the Brown-Erd\H{o}s-S\'os conjecture up to an absolute additive constant.

\end{abstract}

\section{Introduction}\label{sec:intro}

\subsection{Background and previous results}

Some of the most well studied problems in extremal combinatorics are those asking which objects are
guaranteed to appear in ``dense'' objects. Among notable examples are Roth's Theorem \cite{Roth} on $3$-term arithmetic progressions
in dense sets of integers, and the K\H{o}v\'ari-S\'os-Tur\'an Theorem \cite{KST} on bipartite subgraphs of dense graphs.
In this paper we consider a question raised by Brown, Erd\H{o}s and S\'os in 1973 \cite{BES,BES1}, which is one of the most famous open problems of this type.

Given an integer $e\geq 3$, one would expect a dense $3$-uniform hypergraph ($3$-graph for short) to contain $e$ edges spanned by a small number of vertices. To quantify this, let {\em $(v,e)$-configuration} denote a set of $e$ edges spanned by at most $v$ vertices. The Brown--Erd\H{o}s--S\'os Conjecture (BESC) states that for every fixed $e \geq 3$ and all large enough $n$, every $3$-graph with $\Omega(n^2)$
edges contains an $(e+3,e)$-configuration.
Despite a lot of effort over the past 50 years,
the BESC is only known
to hold for $e=3$, due to a result of Ruzsa and Szemer\'edi \cite{RSz}.

Since even the $e=4$ case of the BESC seems hopeless, it is natural to try to prove approximate versions of the conjecture, namely that $3$-graphs with $\Omega(n^2)$ edges contain $(e+f(e),e)$-configurations,
for some slowly growing function $f$. The first result of the above type was obtained by S\'{a}rk\"{o}zy and Selkow \cite{Sarkozy_Selkow} who showed that every $3$-graph
with $\Omega(n^2)$ edges contains for every fixed $e$ an $(e+2+\lfloor \log_2 e\rfloor,e)$-configuration. This was improved by Solymosi and Solymosi \cite{SolymosiSolymosi} for the special case $e=10$ from $15$ to $14$ vertices. A general asymptotic improvement of the result of $\cite{SolymosiSolymosi}$ was obtained recently by Conlon, Gishboliner, Levanzov and Shapira \cite{CGLS}, who proved the existence of $(e+O(\log e/\log\log e),e)$-configurations.

Besides its intrinsic interest, the BESC turned out to be one of the most influential problems in extremal combinatorics.
For example, the proof of the case $e=3$ \cite{RSz} was one of the first applications of Szemer\'edi's regularity lemma \cite{Sz},
and further introduced the famous graph removal lemma. One of the main motivations for the development of the celebrated hypergraph
regularity method \cite{Gowers,NRS,RS1,RS2,Tao} was the hope that it will lead to a resolution of BESC. While this did not materialize,
the hypergraph regularity method was instrumental in the latest works \cite{CGLS,SolymosiSolymosi}. However, although the above
proofs rely on highly involved applications of the hypergraph regularity method, it appears that the following natural approximate version of the BESC
is beyond their reach.

\begin{conjecture}[Constant deficiency BESC]\label{conj:weakBES}
There is an absolute constant $d$ so that for every $e$ and every large enough $n$, every $3$-graph with $\Omega(n^2)$
edges contains an $(e+d,e)$-configuration.
\end{conjecture}

\subsection{A new approach for Conjecture \ref{conj:weakBES}}

Our aim in this paper is to reduce Conjecture \ref{conj:weakBES} to a problem involving graphs.
Let us denote by $\text{ex}(n,H)$ the maximum number of edges in an $n$ vertex graph not containing
a copy of $H$ as a subgraph. The K\H{o}v\'ari-S\'os-Tur\'an Theorem \cite{KST} which we mentioned above,
states that for every fixed $t \leq s$, we have $\text{ex}(n,K_{s,t})=O(n^{2-1/t})$ where $K_{s,t}$
is the complete bipartite graph with parts of size $t$ and $s$.
This bound is known to be tight for large $s$, see \cite{Bukh} for recent progress and references.
One of the main research directions in extremal graph theory is to obtain better bounds for sparser bipartite graphs. One such problem
was raised by Erd\H{o}s \cite{Erdos}, who conjectured that if $H$ is a $t$-degenerate bipartite graph then
$\text{ex}(n,H)=O(n^{2-1/t})$. While there are some approximate results towards this conjecture \cite{AKS,Furedi, GrzesikJanzerNagy, Janzer},
the question is open even for $t=2$. Note that in general, the conjectured bound $O(n^{2-1/t})$ for $t$-degenerate
bipartite graphs cannot be improved since the aforementioned $K_{s,t}$ is $t$-degenerate. In particular,
the bound is tight for every $t$-degenerate $H$ which contains a copy of $K_{s,t}$.
In light of this, Conlon \cite{Conlon} conjectured that if we assume that a $t$-degenerate bipartite graph $H$ has no $K_{t,t}$ then we have $\text{ex}(n,H)=O(n^{2-1/t-\delta})$ for some $\delta=\delta(H)>0$. Lending plausibility to this conjecture, Sudakov and Tomon~\cite{ST} showed that if all vertices in one of the parts of $H$ have degree at most $t$ but $H$ has no $K_{t,t}$ then $\text{ex}(n,H)=o(n^{2-1/t})$.
For $t=2$ Conlon's conjecture can be stated as:
\begin{conjecture}[Conlon~\cite{Conlon}]\label{conj:conlon}
For every $2$-degenerate $C_4$-free bipartite graph $H$ there exists a constant $\delta=\delta(H)>0$ such that $$\text{ex}(n,H)=O(n^{3/2-\delta})\;.$$
\end{conjecture}
\noindent
There are several results supporting Conjecture \ref{conj:conlon}. For example, Conlon and Lee~\cite{CL} proved that
if $H$ is a bipartite graph so that each vertex in one of $H$'s sides has maximum degree $2$ (such a graph is clearly $2$-degenerate)
and $H$ is $C_4$-free then $\text{ex}(n,H)=O(n^{3/2-\delta})$ for some $\delta=\delta(H)>0$. Further results in this direction were obtained in \cite{ConlonJanzerLee, Janzer1}.

Let $\mathcal{H}_{k,t}$ be the family of $2$-degenerate graphs on $k$ vertices and $2k-t$ edges.
We raise the following weaker version of Conjecture \ref{conj:conlon}.

\begin{conjecture}\label{conj:weak}
There are absolute constants $t,k_0$ such that for every $k \geq k_0$ and large enough $n$,
every graph with $\Omega(n^{3/2})$ edges contains
a copy of {\bf some} $H \in \mathcal{H}_{k,t}$.
\end{conjecture}

Let us briefly explain why Conjecture \ref{conj:weak} is indeed weaker than Conjecture \ref{conj:conlon}.
It is not hard to see that for every $t$ and large enough $k$, the family $\mathcal{H}_{k,t}$ contains
$C_{4}$-free graphs (see Claim \ref{prop:girth}).
Conjecture \ref{conj:conlon} then states that if $G$ has $\Omega(n^{3/2})$ edges then $G$ should contain
a copy of {\bf every} $H \in \mathcal{H}_{k,t}$ which is $C_4$-free, while Conjecture \ref{conj:weak} only asks
$G$ to contain a copy of some $H \in \mathcal{H}_{k,t}$.
Note also that Conjecture \ref{conj:weak} is weaker than the statement that for every $k \geq k_0$ we have $\text{ex}(n,H)=o(n^{3/2})$ for some $H \in \mathcal{H}_{k,t}$, which is itself weaker than Conjecture \ref{conj:conlon}.


Our main result in this paper is the following alternative approach for resolving Conjecture \ref{conj:weakBES}.

\begin{theorem}\label{thm:main}
Conjecture~\ref{conj:weak} implies Conjecture \ref{conj:weakBES}.
\end{theorem}


Before turning to the proof of Theorem \ref{thm:main}, we mention that it might very well be the case
that in Conjecture \ref{conj:weak} we can replace the lower bound $\Omega(n^{3/2})$ by $\Omega(n^{3/2-\delta})$
for some $\delta=\delta(k)>0$. Indeed, this bound is implied by Conjecture \ref{conj:conlon}.
It is not hard to see that in this case the proof of Theorem \ref{thm:main} would give that for some absolute constant $d$ and for every $e$
there is $\varepsilon=\varepsilon(e)>0$ so that one can find $(e+d,e)$-configurations in every $3$-graph with $n^{2-\varepsilon}$ edges.
Such a result would be an approximate version of a conjecture suggested by Gowers and Long \cite{GowersLong}, stating that $3$-graphs with
$n^{2-\varepsilon}$ edges contain $(e+4,e)$-configurations.

\section{Proof of Theorem \ref{thm:main}}

To avoid confusion, we will refer to edges of a $3$-graph as hyperedges.
Fix $e\geq 3$ and let $\G$ be a $3$-graph with $n$ vertices and $\Omega(n^2)$ hyperedges.
We will rely on the well known observation that in the context of the BESC one can assume that $\G$ is linear and $3$-partite
on vertex sets $(\A,\B,\C)$. We now apply a variant of the construction of Solymosi and Solymosi~\cite{SolymosiSolymosi}. Given $\G$, define an auxiliary bipartite multigraph $G'$ as follows. Set $V(G')=(A,B)$ where $A=\binom{\A}{2}$ and $B=\binom{\B}{2}$. For two vertices $\{a_1,a_2\}\in A$ and $\{b_1,b_2\}\in B$ put an edge between them if there is a $c\in \C$ so that $a_1b_1c$ and $a_2b_2c$ are hyperedges of $\G$, and (independently) put an edge between them if there is a $c'\in \C$ such that $a_1b_2c'$ and $a_2b_1c'$ are hyperedges of $\G$. Since $\G$ is linear, each pair of vertices in $G'$ are connected by at most $2$ edges.
If we let $d(c)$ denote
the degree of a vertex $c\in \C$ in $\G$ then
$$|E(G')|=\sum_{c\in \C}\binom{d(c)}{2}\geq |\C| \binom{\frac{1}{|\C|}\sum_{c\in \C} d(c)}{2}=|\C|\binom{|E(\G)|/|\C|}{2}\geq \frac{|E(\G)|^2}{4|\C|}.
$$
\noindent
Since $e(\G)=\Omega(n^2)$, $|\C|\leq n$, and $|V(G')|\leq n^2$, we obtain  $|E(G')|=\Omega(|V(G')|^{3/2})$.
Since, as noted above, each pair of vertices in $G'$ are connected by at most $2$ edges, $G'$ has a simple subgraph $G$ which also contains $\Omega(|V(G)|^{3/2})$ edges.
Therefore, if $k_0$ and $t$ are the constants from Conjecture \ref{conj:weak} and $n$ is large enough, then we may assume the following.

\begin{observation}\label{clm:F}
For every $k_0\leq k \leq e$, the graph $G$ contains a $2$-degenerate bipartite graph $F$ on $k$ vertices with at least $2k-t$ edges.
\end{observation}

We would now like to understand what kind of $(v,e)$-configuration in $\G$ we get by ``unpacking'' each of the graphs $F$ in Observation \ref{clm:F}. Optimistically, if $v_1,\ldots,v_k$ is the ordering of
$V(F)$ certifying its $2$-degeneracy, then every time we add a vertex $v_i$ to $v_1,\ldots,v_{i-1}$ of degree $2$ to the previous vertices, we expect to get
$4$ new vertices in $\G$; these are $c_1,c_2$ and either $a_1,a_2$ (if $v_i \in A$) or $b_1,b_2$ (if $v_i \in B$). We also expect to get $4$ new hyperedges in $\G$; these are the $4$ hyperedges that correspond to the $2$ new edges in $G$ that connect $v_i$ to $2$ of the vertices $v_1,\ldots,v_{i-1}$.
If this holds for all but a bounded number of $F$'s vertices, then we will get a $(4k,4k-O_k(1))$ configuration, hence taking $k \approx e/4$ would finish the proof.
Unfortunately, we do not know how to prove such a statement, since in certain cases (see below) some of the $4$ vertices/hyperedges might have already appeared when adding one of the previous vertices $v_j$. Instead, the main idea in Lemma \ref{lem:key} below is to show that $F$ gives rise to a $(e'+d,e')$-configuration, so that if $e'$ is not very close to $4k$ (as in the optimistic analysis above) then we have $d\leq 0$. It is then easy to show how repeated applications of Lemma \ref{lem:key} give Theorem \ref{thm:main}. In what follows $G$ and $\G$ are those
we discussed above.

\begin{lemma}\label{lem:key}
Let $k\geq t\geq 4$ be integers, and suppose $F$ is a $2$-degenerate subgraph of $G$ with $k$ vertices and $2k-t$ edges.
Then $\G$ contains a subgraph $\F$ such that
\begin{enumerate}[label=(\roman*)]
	\item[(1)] $|V(\F)|-4t \leq |E(\F)| \leq 4k $, and
	\item[(2)] Either $|E(\F)|\geq 4k-10^4t^3$ or $|E(\F)|\geq |V(\F)|>0$.
\end{enumerate}
\end{lemma}

We first derive Theorem~\ref{thm:main} from Lemma~\ref{lem:key}.
Assuming Conjecture~\ref{conj:weak} holds with constants $t,k_0$ we show that
Conjecture \ref{conj:weakBES} holds with $d=\max\{24k_0,3(4t+10^4t^3)\}$.
Indeed, we claim that for every $0 \leq e' \leq e $ we can find $e'$ hyperedges in $\G$ spanned by at most $e'+d$ vertices.
If $e' \leq \max\{8k_0,4t+10^4t^3\}$, we just take $e'$ arbitrary hyperedges from $\G$.
For larger $e'$ we apply Lemma~\ref{lem:key} with the above $t$ and with $k=\lfloor e'/4\rfloor \geq k_0$ (by Observation \ref{clm:F} we know that $G$ contains an $F$ with these parameters).
If the lemma returns a configuration $\F'$ whose number of edges satisfies $e'-10^4t^3-4 \leq |E(\F')|\leq e'$ (and is on at most $e'+4t$ vertices),
we just add to $\F'$ arbitrarily chosen $e'-|E(\F')|\leq 10^4t^3+4$ hyperedges to get a set of $e'$ edges on at most $e'+d$ vertices.
Otherwise, we have $|E(\F')|\geq |V(\F')|>0$ so we can remove $\F'$ from $\G$ and then restart the process with $e''=e'-|E(\F')|$ (the $3$-graph $\G \setminus \F$ still
has $\Omega(n^2)$ hyperedges assuming $n$ is large). We will obtain
a set $\F''$ of $e''$ hyperedges on at most $e''+d$ vertices, and can then return $\F'' \cup \F'$ as the set
of $e'$ hyperedges on at most $e'+d$ vertices.

\begin{proof}[Proof of Lemma \ref{lem:key}]
Suppose $G$ contains a subgraph $F$ as above. Let $v_1,\dots,v_k$ be the vertices of $F$ in the order that certifies its $2$-degeneracy. For each $i\in[k]$ let $F_i=F[v_1,\dots,v_i]$ be the induced subgraph on the first $i$ vertices. Let $\F_i\subseteq \G$ be a subgraph of $\G$ that corresponds to $F_i$. That is,
$$V(\F_i)\cap (\A\cup \B)=\{p\in \A\cup \B\colon \{p,q\}\in V(F_i) \text{ for some } q\in \A\cup \B\},$$
and for every edge $uv$ of $F_i$, where $u=\{a_1,a_2\}$ and $b=\{b_1,b_2\}$ let $c\in \C$ be the (unique) vertex certifying that $uv\in E(F)$ (in particular $\{a_1b_1c,a_2b_2c\}\subseteq E(\G)$ or $\{a_1b_2c,a_2b_1c\}\subseteq E(\G)$). We include $c$ in $V(\F_i)$ and the corresponding pair of hyperedges 
in $E(\F_i)$, and applying the same procedure for each edge of $F_i$ we take the union of the resulting hyperedges.

\paragraph{Proof of assertion $(1)$:}
Initially we have a graph $F_0:=(\emptyset,\emptyset)$ with $0$ edges and vertices.
Given some $i\in[k]$, let $F^-:=F_{i-1}$ and $\F^-:=\F_{i-1}$. Suppose without loss of generality that $v_i\in A$, that is, $v:=v_i$ corresponds to a pair $\{a_1,a_2\}\in \binom{\A}{2}$. 
Let $d(v)$ denote the degree of $v$ in $F_i$, by our assumptions we have $d(v)\leq 2$. Let $\Delta_E(i):=|E(\F_i)\setminus E(\F^-)|$ and $\Delta_V(i):=|V(\F_i)\setminus V(\F^-)|$.

Note that $$0\leq \Delta_E(i)\leq 2 d(v)\leq 4,$$ which, summing over all $i$ gives the inequality $|E(\F)| \leq 4k$ stated in assertion $(1)$.
To prove the second inequality, we need to consider the degree of $v$: if $d(v)=2$, let us call $v$ a \emph{regular} vertex, otherwise (if $d(v)$ is $0$ or $1$) we say that $v$ is \emph{singular}. Accordingly, we are speaking of a regular or singular step $i$.
A crucial observation is that since $F$ is $2$-degenerate and has $2k-t$ edges, then the total number of singular steps is at most $2t$.

Suppose first that $v$ is regular, and let $u$ and $w$ be the two neighbours of $v$ in $F^-$. Let $u$ and $w$ correspond to $\{b_1,b_2\}\in \binom{\B}{2}$ and $\{b_3,b_4\}\in \binom{\B}{2}$ respectively, with  $\{b_1,b_2\}\neq \{b_3,b_4\}$ (note that some individual $b_1,b_2,b_3,b_4$ may coincide).
Furthermore, we have vertices $c_1,c_2\in \C$ such that (after relabelling) $a_1b_1c_1$, $a_2b_2c_1$, $a_1b_3c_2$ and $a_2b_4c_2$ are hyperedges of $\F_i$. Note that we must have $c_1\neq c_2$ for otherwise, by linearity of $\G$, we would have $b_1=b_3$ and $b_2=b_4$, and so $\{b_1,b_2\}=\{b_3,b_4\}$.
Since all other hyperedges of $\F_i$ were already contained in $\F^-$, we have
\begin{equation}\label{eqdeltaE}
E(\F_i)\setminus E(\F^-)\subseteq \{a_1b_1c_1, a_2b_2c_1, a_1b_3c_2, a_2b_4c_2\}\;.
\end{equation}
Similarly, $V(\F_i)\setminus V(\F^-)\subseteq\{a_1,a_2,c_1,c_2\}$, and so we also have $0 \leq \Delta_V(i) \leq 4$.

We now claim that $\Delta_V(i)\leq \Delta_E(i)$, and that in fact $\Delta_V(i)< \Delta_E(i)$ when $\Delta_E(i)\in \{1,2,3\}$ (this will be used in the proof of assertion $(2)$). Indeed, if $\Delta_E(i)=4$, then there is nothing to prove since $\Delta_V(i)\leq 4$. If $\Delta_E(i)=3$, then without loss of generality the hyperedge $a_1b_1c_1$ was already contained in $\F^-$. Hence, $\{a_1,c_1\}\subseteq V(\F^-)$, implying $\Delta_V(i)\leq 2$. Similarly, if $\Delta_E(i)=2$, we have $\Delta_V(i)\leq 1$ (if $a_1b_1c_1$ and $a_2b_2c_1$ were in $\F^-$ then only $c_2$ can be a new vertex, and if $a_1b_1c_1$ and one of the hyperedges containing $c_2$ were already in $\F^-$ then only $a_2$ can be a new vertex), and if $\Delta_E(i)=1$, then $\Delta_V(i)= 0$ (if only $a_1b_1c_1$ is a new hyperedge then $c_1$ was added with $a_2b_2c_1$ and $a_1$ was added with $a_1b_3c_2$.). Finally, if $\Delta_E(i)=0$ then $\{a_1, a_2,c_1,c_2\}\subseteq V(\F^-)$ so $\Delta_V(i)=0$. So, we obtain $|E(\F_i)|-|V(\F_i)|\geq |E(\F_{i-1})|-|V(\F_{i-1})|$.

If $v$ is singular, a similar case analysis shows that  $|E(\F_i)|-|E(\F_{i-1})|\geq |V(\F_i)|-|V(\F_{i-1})|-2$. Since there are at most $2t$ singular steps in total, summing over all $i$ yields $|E(\F)|\geq |V(\F)|-4t$ as desired.

\paragraph{Proof of assertion $(2)$:} In order to prove the second assertion we need to study the above process in more detail.

Suppose a step $i$ is regular. If $\Delta_E(i)=0$ we call it a \emph{$0$-step}, if $\Delta_E(i)=\Delta_V(i)=4$ we say this is a \emph{$4$-step}.
If we have $\Delta_V(i)<\Delta_E(i)$, then we call this step a \emph{good} regular step. Note that by the argument
in the paragraph following (\ref{eqdeltaE}), every regular step which is not a $0$-step or a $4$-step is a good step.
Note also that at each good regular step the difference $|E(\F_i)|-|V(\F_i)|$ strictly increases and, as we have seen in the proof of $(1)$, this difference decreases only at singular steps, in which it decreases by at most $2$. Hence, if the total number of good regular steps is at least $4t$ we would have $|E(\F)|\geq |V(F)|>0$ as needed. So let us assume for the rest of the proof that we have fewer than $4t$ good regular steps. Let us say that a (regular or singular) step is \emph{good} if it is either good regular in the above sense or singular. So, the total number of good steps is less than $6t$.

If the number of $0$-steps is at most $s:=6t(12t+2)^2$, then all but $s+6t$ of the steps are $4$-steps and so we have $|E(\F)|\geq 4k-4(s+6t)\geq 4k-10^4t^3$ as needed. So suppose towards contradiction that this is not the case, i.e., that the number of $0$-steps is greater than $s$.
We will now show that this means that the total number of good and steps is at least $6t$, contradicting the statement made in the previous paragraph.

We say that a vertex $c\in \C$ is \emph{involved} in step $i$ (or equivalently, step $i$ \emph{involves} $c$) if $c$ plays the role of either $c_1$ or $c_2$ in the extension of $\F_{i-1}$ to $\F_{i}$ described above. Note that each regular step involves precisely two vertices of $\C$. Similarly, we say that a hyperedge $e\in E(\G)$ is involved in step $i$ if it plays the role of one of the hyperedges arising in the extension of $\F_{i-1}$ to $\F_{i}$ (we stress that this is regardless of whether $e$ had already been contained in $\F_{i-1}$).
\begin{observation}\label{obs:pair}
A pair of hyperedges $e_1=a_1b_1c$ and $e_2=a_2b_2c$, where $a_1,a_2\in A, b_1,b_2\in B, c\in C$, can simultaneously be involved in at most one step.
\end{observation}
\noindent
Indeed, for every step $i$ involving both hyperedges there must be vertices $u,w\in V(F)$ with $u=\{a_1,a_2\}$ and $w=\{b_1,b_2\}$ such that one of $u$ and $w$ is the vertex $v_i$ and the other is $v_j$ for some $j<i$.

We now claim that every $0$-step involving some vertex $c\in \C$ must be preceded by a good step involving $c$.
Indeed, suppose that $c$ is involved in a $0$-step at time $i$. Suppose that $v_i$ represents some $\{a_1,a_2\}\in \binom{\A}{2}$ with $E(F_{i})\setminus E(F_{i-1})=\{v_iu, v_iw\}$ for some $u,w\in V(F)$ representing $\{b_1,b_2\}, \{b_3,b_4\}\in \binom{\B}{2}$ respectively (the case when $v_i\in B$ is identical), and that the hyperedges of $\G$ certifying that $\{v_iu, v_iw\}\subseteq E(F_i)$ (after relabelling) are $\{a_1b_1c,a_2b_2c,a_1b_3c',a_2b_4c'\}$ for some $c'\in \C$ (and note that since this is a $0$-step, all these hyperedges are already contained in $E(\F_{i-1})$). Let $j_1$ be the first step involving the hyperedge $e_1:=a_1b_1c$, i.e. $j_1<i$ is the unique $j$ such that $e_1\in E(\F_j)\setminus E(\F_{j-1})$. Let $j_2$ be defined analogously with respect to $e_2:=a_2b_2c$. If step $j_1$ or $j_2$ are singular, we have proved the claim (since singular steps are good by definition). So, let us assume they are both regular. If $j_1\neq j_2$ then at time $\max(j_1,j_2)$ (say, this is $j_2$) we have a good step involving $c$, since $\Delta_V(j_2)<4$ yet $\Delta_E(j_2)\geq 1$, so this cannot be a $0$-step or a $4$-step and thus must be a good step. On the other hand we cannot have $j_1=j_2$ since that would mean both $e_1$ and $e_2$ would be involved in two different steps, contradicting Observation~\ref{obs:pair}. This proves the above claim.

Now let $\Z\subseteq \C$ be the set of all vertices in $\C$ involved in $0$-steps. Suppose first that $|\Z|>12t$. Then, as for every $z\in \Z$ each $0$-step involving $z$ is preceded by a good step also involving $z$, the number of vertices of $\C$ involved in good steps is greater than $12t$. Since every step involves at most $2$ vertices of $\C$, we obtain that the total number of good steps is greater than $6t$, as needed.

So, let us assume that $|\Z|\leq 12t$. Then, by pigeonhole, some $z\in \Z$ was involved in at least $(2s)/(12t)=(12t+2)^2$ of the $0$-steps. This implies that $z$ must be contained in at least $12t+2$ hyperedges of $\F$, as each $0$-step involving $z$ involves two hyperedges containing $z$, and no such pair may be involved twice by Observation~\ref{obs:pair}.

Let now $J\subset [k]$ be the set of all $j\in [k]$ such that at step $j$ for some hyperedge $e\in \F$ with $z\in e$ we have $e\in E(\F_j)\setminus E(\F_{j-1})$. Since at any given step $j$ we can have at most $2$ such hyperedges $e$, we have $|J|\geq (12t+2)/2= 6t+1$. On the other hand for every step in $j\in J$ except $j_0=\min J$ we have $\Delta_V(j)<4$, since $z\in \F_{j_0}$, and $\Delta_E(j)>0$, by definition of $J$. This means that each of these $|J|-1\geq 6t$ steps is not a $0$-step or a $4$-step, and therefore must be a good step.

We have thus shown that if the number of $0$-steps is at most $s$ then the number of good steps is at least $6t$, which completes the proof of the lemma.
\end{proof}

\section{$C_4$-free graphs in $\mathcal{H}_{k,t}$}

We say that a graph is \emph{exactly-$(2,t)$-degenerate} if it can be obtained from a set of $t$ isolated vertices by repeatedly adding new vertices of degree exactly $2$. Note that every exactly-$(2,t)$-degenerate graph belongs to $\mathcal{H}_{k,t}$.
The following claim shows that $\mathcal{H}_{k,t}$ contains not only $C_4$-free graphs, but in fact graphs of arbitrary large girth.

\begin{claim}\label{prop:girth}
For every $g$ there is $t=t(g)$ so that for every $k\geq t$, there is a $k$-vertex exactly-$(2,t)$-degenerate bipartite graph of girth at least $g$.
\end{claim}

\begin{proof}
We claim that starting with an independent set of size $t=t(g)$ (to be chosen later),
we can repeatedly add vertices so that each $k$-vertex graph in the sequence
is exactly-$(2,t)$-degenerate, bipartite, of girth at least $g$, and in addition satisfies the following two conditions:
$(i)$ it has maximum degree at most $8$ and $(ii)$ it has a bipartition into two set of sizes $\lceil k/2 \rceil$ and $\lfloor k/2 \rfloor$.
The initial independent set under a balanced bipartition clearly satisfies these two conditions, so let us show how to add a vertex and maintain them.
Suppose the graph has $k-1$ vertices and bipartition into sets $A,B$ satisfying $|A| \leq |B|$.
Since it has maximum degree at most $8$, it contains $O(k)$ pairs of vertices connected by a path of length at most $g-2$.
Since the average degree of the vertices in $B$ is less than $4$, at least half the vertices have degree at most $7$. Hence, at least $\binom{(k-1)/4}{2}\geq \frac{k^2}{50}$ of the pairs of vertices in $B$ both have degree at most $7$.
Assuming $t$ is large enough so that $k\geq t$ satisfies $\frac{k^2}{50} -O(k) > 1$, we thus have a pair of vertices $u,v \in B$ so that both of them have degree at most $7$ and there is no path of length at most $g-2$ connecting them. Hence, we can add a new vertex to $A$ and connect it to $u$ and $v$.
\end{proof}

\medskip

\noindent{\bf Acknowledgement:} We would like to thank David Conlon for useful discussions.

\end{document}